\documentclass{amsart}

\usepackage{amsmath,amssymb,amsthm}

\numberwithin{equation}{section}
\newcommand{\be}{\begin{equation}}
\newcommand{\ee}{\end{equation}}

\newcommand{\pd}{\partial}
\newcommand{\R}{\mathbb R}

\newcommand{\C}{\mathcal C}

\newcommand{\co}{\colon}
\newcommand{\D}{\mathcal D}

\newcommand{\ga}{\gamma}
\newcommand{\la}{\lambda}
\newcommand{\ep}{\varepsilon}

\newcommand{\ovr}{\overrightarrow}

\newcommand\<{\langle}
\renewcommand\>{\rangle}

\newcommand{\grad}{\operatorname{grad}}

\newtheorem{lemma}{Lemma}[section]
\newtheorem{proposition}[lemma]{Proposition}
\newtheorem{theorem}{Theorem}

\theoremstyle{definition}

\newtheorem*{notation}{Notation}

\theoremstyle{remark}
\newtheorem{remark}[lemma]{Remark}

\begin{document}

\title[Distance difference functions]
{Distance difference functions on non-convex boundaries
of Riemannian manifolds}

\author{Sergei Ivanov}
\address{St.~Petersburg Department of Steklov Mathematical Institute of
Russian Academy of Sciences,
Fontanka 27, St.Petersburg 191023, Russia}
\address{Saint Petersburg State University, 
7/9 Universitetskaya emb., St. Petersburg 199034, Russia
}
\email{svivanov@pdmi.ras.ru}

\thanks{Research is supported by the Russian Science Foundation grant 16-11-10039}



\begin{abstract}
We show that a complete Riemannian manifold with boundary
is uniquely determined, up to an isometry,
by its distance difference representation on the boundary.
Unlike previously known results, we do not impose any
restrictions on the boundary.
\end{abstract}

\maketitle

\section{Introduction}

Let $M=(M,g)$ be a complete, connected, $C^\infty$
Riemannian manifold with boundary $F=\pd M\ne\emptyset$.
We study the following structure introduced by Lassas and Saksala \cite{LS}.
For every $x\in M$ consider the \textit{distance difference function}
$$
 D_x\co F\times F \to\R
$$
defined by
\be\label{e:ddef}
 D_x(y,z) = d_M(x,y)-d_M(x,z), \qquad y,z\in F ,
\ee
where $d_M$ is the arclength distance in~$M$.
Note that for every $x_1,x_2\in M$, the distance $d_M(x_1,x_2)$
is realized by a shortest path
which is a $C^1$ curve, see e.g.~\cite{AA}.
The boundary is not assumed convex, thus
a shortest path can touch the boundary, bend along it, etc.

We regard the collection of functions $D_x$, $x\in M$,
as a map 
$$
\D\co M\to\C(F\times F)
$$
given by
\be\label{e:ddef2}
\D(x)=D_x , \qquad x\in M .
\ee
Here $\C(F\times F)$ denotes the space of all continuous functions
on $F\times F$.
The map $\D$ is called the \textit{distance difference representation} of $M$,
and the set $\D(M)\subset\C(F\times F)$
is called the \textit{distance difference data} \cite{LS}
or \textit{travel time difference data} \cite{HS}.
In this paper we use the former term.

The main result of this paper tells that the distance difference data $\D(M)$
determine $M$ uniquely up to a Riemannian isometry.
See Theorem \ref{t:whole} below for the precise statement.
This improves results from \cite{HS} and \cite{I20}
where similar theorems were obtained under additional
assumptions on the boundary.
Moreover in Theorem \ref{t:part} we show that the geometry of an arbitrary open region $U\subset M$
is determined by the partial distance difference data $\D(U)$.

The intuition behind the problem is the following.
Imagine that $M$ is some material object of interest,
for example the Earth, and $g$ represents the speed of
wave propagation in $M$.
A point $x\in M$ can be a spontaneous spherical wave source
(for example, think of microseismic events in the Earth's crust).
An observer measures arrival times
of the wave at a dense set of points on the surface.
Since the time of the event in the interior is unknown,
the information obtained from the measurement is precisely
the same as that provided by the function $D_x$.
Knowing the set $\D(M)$ means that such information
is collected from a dense set of points in the interior,
and the goal is to learn the geometry of $M$ from these data.
See \cite{LS,HS} for more detailed discussion of applications and \cite{KKL}
for applications of ordinary distance representations.
i.e., plain distance functions $d_M(x,\cdot)$ rather than
the differences.


\subsection*{Previous results}
Lassas and Saksala \cite{LS} proved the unique determination of $M$
in a different setting where $M$ is compact and has no boundary,
and the ``observation domain'' $F$ is an open subset of $M$
rather than the boundary.
In \cite{I20} this is generalized to the case
of a complete but possibly non-compact manifold $M$
and partial distance difference data $\D(U)$ where $U\subset M$
is an open set whose geometry is to be determined.

The case when $F=\pd M$ turns out to be more difficult.
It is partly addressed in \cite{HS} and \cite{I20}
where the unique determination of the manifold is proved
under various additional assumptions on the boundary.
The assumption in \cite{I20} is the nowhere concavity
of the boundary and in \cite{HS} it is a less restrictive
``visibility'' condition.
We emphasize that in this paper we do not make additional
assumptions on the boundary.

The similar problem about boundary distance data
(without the ``difference'' part) was solved earlier,
see \cite{KKL,Kur}.  See also \cite{KLU} for similar problems
in the Lorentzian setting.

\subsection*{Statement of the results}
We begin with a global determination result:

\begin{theorem}\label{t:whole}
Let $M=(M,g)$ and $M'=(M',g')$
be complete, connected
Riemannian $n$-manifolds ($n\ge 2$) with a common
boundary $F=\pd M=\pd M'\ne\emptyset$.
Assume that the distance difference data of $M$ and $M'$
coincide as subsets of $\C(F\times F)$,
i.e.\ $\D(M)=\D'(M')$ where $\D$ and $\D'$ are
their respective distance difference representations,
see \eqref{e:ddef} and \eqref{e:ddef2}.

Then $M$ and $M'$ are isometric via a Riemannian isometry that fixes $F$.
\end{theorem}

The words ``common boundary'' in the theorem require clarification.
Their precise meaning is the following: $F$ is a topological manifold
and it is identified with $\pd M$ and $\pd M'$ by means of some
homeomorphisms. In other words, we assume that $M$ and $M'$
induce the same topology on $F$ but do not assume that they
induce the same differential structure and metric.

Like in \cite{I20}, we actually prove a more general result on
unique determination of local geometry from the corresponding
partial data:

\begin{theorem}\label{t:part}
Let $M=(M,g)$ and $M'=(M',g')$ be complete, connected
Riemannian $n$-manifolds ($n\ge 2$)
with a common boundary $F=\pd M=\pd M'\ne\emptyset$.
Let $\D$ and $\D'$ denote the distance difference representations
of $M$ and $M'$ in $\C(F\times F)$, see \eqref{e:ddef} and \eqref{e:ddef2}.
Let $U\subset M$ and $U'\subset M'$ be open sets such that
\be\label{e:sameset}
 \D(U) = \D'(U') .
\ee
Then there exists a Riemannian isometry 
$\phi\co U\to U'$
(with respect to the metrics $g$ and $g'$)
such that $\phi|_{U\cap F}$ is the identity.
\end{theorem}

Note that Theorem  \ref{t:whole} is a special case 
of Theorem \ref{t:part} for $U=M$ and ${U'=M'}$.
The proof of Theorem \ref{t:part} occupies the rest of the paper.
It builds upon results from  \cite{I20} and \cite{KKL}.
In particular one of the important ingredients of the proof is
Proposition \ref{p:bilip} borrowed from \cite{I20}.

The paper is organized as follows.
In Section \ref{sec:prelim} we fix notation,
collect preliminaries and construct a candidate
for the isometry $\phi$.
In Section \ref{sec:nearest} we find a region on
the boundary where a given distance function is regular,
see Lemma \ref{l:no-cut}.
This step of the proof is essentially borrowed from \cite{KKL}.
In Section \ref{sec:shortest}
we show that certain minimizing
geodesics in $M$ are mapped by $\phi$ to geodesics in $M'$
(but not necessarily preserving the arc length).
Similar proof steps can be found in \cite{LS,HS,I20}
but the details in each case are different.
Finally in Section \ref{sec:final} we prove the key Lemma \ref{l:dphi},
which in a sense reconstructs the metric
tensor at a point,
and deduce the theorems.
The arguments in Section \ref{sec:final}
are similar to those in \cite{I20} with some
modifications.

\subsection*{Remarks on regularity}
The Riemannian manifolds in this paper are $C^\infty$.
This is different from \cite{I20} where the proof of the main result
requires only sectional curvature bounds
and works for Alexandrov spaces as well.
More regularity is needed in Proposition~\ref{p:bilip}
whose proof in \cite{I20} depends
on smooth extension of the metric beyond the boundary,
and in Lemma \ref{l:no-cut} where we rely on properties of cut loci.

It is plausible that the result holds under weaker regularity
assumptions such as uniform bounds on the sectional curvature and
the second fundamental form of the boundary.
Such an improvement would
imply stability of manifold determination with respect
to Gromov-Hausdorff topology, cf.~\cite[Proposition 6.4]{I20}.

 \section{Preliminaries and notation}
\label{sec:prelim}

Let $M,g,\D,U,M',g',\D',U'$ be as in Theorem \ref{t:part}.
We denote by $D_x$ and $D'_x$ the distance difference functions
defined by \eqref{e:ddef} for $M$ and $M'$, resp.
In the sequel a number of lemmas are stated
only for $M$ but they apply to both $M$ and $M'$.

\begin{notation}
For $x\in M$, we denote by $T_xM$ the tangent space of $M$ at $x$
and by $S_xM$ the unit sphere of $T_xM$ (with respect to~$g$).
That is,
$$
S_xM = \{v\in T_xM : \|v\|_g = 1\} .
$$
As usual we write $\<u,v\>$ instead of $g(u,v)$ for $u,v\in T_xM$.

For $x\ne y\in M$, we denote by $[xy]$ a shortest path from $x$ to $y$.
In general, a shortest path is not unique;
we assume that some choice of $[xy]$ is fixed for every pair $x,y$.
By $\ovr{xy}$ we denote the unit tangent vector of $[xy]$ at~$x$.
That is, $\ovr{xy}\in S_xM$ is the initial velocity of the unit-speed
parametrization of $[xy]$.
\end{notation}

We need the following standard implication of the Gauss Lemma:
If $x\ne y\in M$ and the distance function
$d_M(\cdot,y)$ is differentiable at $x$ then the Riemannian gradient
of this function at $x$ is given by
\be\label{e:grad0}
 \grad_g d_M(\cdot,y)|_x = - \ovr{xy}
\ee
Equivalently, the differential of $d_M(\cdot,y)$ at $x$ is given by
\be\label{e:grad}
 d_x d_M(\cdot,y) = -\<\cdot , \ovr{xy} \>_g
\ee
for all $v\in T_xM$.

To see why \eqref{e:grad0} holds
(even in manifolds with boundary),
observe that $d_M(\cdot,y)$ is a 1-Lipschitz function and
it decays with speed 1 along $[xy]$.
Hence $-\ovr{xy}$ is the direction of maximum growth
of this function
and the growth rate is~1, therefore it is the gradient.

\medskip

We regard the target space $\C(F\times F)$ of $\D$ 
with the sup-norm distance,
$$
 \|u-v\| = \sup_{x,y\in F} |u(x,y)-v(x,y)|, \qquad u,v\in \C(F\times F).
$$
If $F$ is not compact then this distance can
attain infinite values.
However the distance between functions from $\D(M)$ is always finite.
Indeed, the triangle inequality for $d_M$ implies that
$$
 \|\D(x)-\D(y)\| \le 2d_M(x,y) < \infty
$$
for all $x,y\in M$. 
This inequality also shows that $\D$ is a 2-Lipschitz map.

We need the following result from \cite{I20}.

\begin{proposition}[{\cite[Proposition 7.1]{I20}}]\label{p:bilip}
The map $\D\co M\to\C(F\times F)$ is a locally bi-Lipschitz
homeomorphism between $M$ and $\D(M)$.
\end{proposition}

Now, applying Proposition \ref{p:bilip} to $M$ and $M'$
and using the assumption \eqref{e:sameset}
of Theorem \ref{t:part}, 
we can define a locally bi-Lipschitz homeomorphism
$
 \phi \co U\to U'
$
by
$$
 \phi = (\D')^{-1}\circ\D|_{U} .
$$
The definition of $\phi$ implies that
\be\label{e:phi0}
 D_x = D'_{\phi(x)}
\ee
for all $x\in U$.
Our ultimate goal is to prove that $\phi$ is a Riemannian isometry.

\section{Nearest and almost nearest boundary points}
\label{sec:nearest}

Since $M$ is complete, all closed balls of the metric $d_M$
are compact, see e.g.\ \cite[Proposition 2.5.22]{BBI}.
Therefore for every $x\in M$
there exists at least one nearest boundary point,
i.e., a point $y\in F$ realizing the minimum of 
the function $d_M(x,\cdot)|_F$.


\begin{lemma}\label{l:phi-nearest}
Let $x\in U$ and let $y\in F$ be a nearest boundary point to $x$ in $M$.
Then $y$ in a nearest boundary point to $\phi(x)$ in $M'$.
\end{lemma}

\begin{proof}
A point $y\in F$ is a nearest boundary point to $x$ if and only if
$$
\forall z\in F \qquad D_x(y,z) \le 0 ,
$$
see \eqref{e:ddef}.
By \eqref{e:phi0}, this relation implies the same one for $D'_{\phi(x)}$
in place of $D_x$.
Hence $y$ in a nearest boundary point to $\phi(x)$ in $M'$.
\end{proof}

Now we can prove that $\phi$ satisfies the last 
requirement of Theorem \ref{t:part}.

\begin{lemma}\label{l:id-on-boundary}
$U\cap F=U'\cap F$ and $\phi|_{U\cap F}$ is the identity map.
\end{lemma}

\begin{proof}
Since $U$ and $U'$ are open subsets of $M$ and $M'$,
they are topological manifolds, possibly with boundaries
$\pd U= U\cap F$ and $\pd U'=U'\cap F$.
Since $\phi$ is a homeomorphism between $U$ and $U'$,
it sends boundary to boundary. Thus 
\be\label{e:id-on-boundary1}
  \phi(U\cap F)=U'\cap F .
\ee
Any point $x\in U\cap F$ is a unique nearest boundary
point to itself.
This and Lemma \ref{l:phi-nearest} imply that
$x$ is a unique nearest boundary point to $\phi(x)$.
Since $\phi(x)\in F$ by \eqref{e:id-on-boundary1}, it follows that $\phi(x)=x$.
Thus $U\cap F=U'\cap F$ and $\phi|_{U\cap F}$ is the identity.
\end{proof}

Let $p\in M\setminus F$ and let $q\in F$
be a nearest boundary point to $p$.
Consider a shortest path $[pq]$.
It meets $F$ only at $q$,
therefore it is a Riemannian geodesic.
Furthermore, the first variation formula implies that $[pq]$
meets $F$ orthogonally at~$q$.
These properties imply that $[pq]$ is a unique shortest path
between $p$ and $q$.

As shown in \cite[Lemma 2.13]{KKL}, $p$ and $q$ are not
conjugate along $[pq]$, hence $p$ is not a cut point of $q$
and therefore the distance function $d_M(\cdot,q)$ is smooth at~$p$.
Since the cut-point relation is closed, similar properties hold
for all boundary points sufficiently close to~$q$.
Namely we have the following lemma.

\begin{lemma}\label{l:no-cut}
Let $p\in M\setminus F$
and let $q\in F$ be a nearest boundary point to $x$.
Then there exists an neighborhood $V\subset F$ of $q$
such that for every $z\in V$ the following holds:
\begin{enumerate}
\item There is a unique shortest paths $[pz]$ in $M$;
\item $[pz]\cap F = \{z\}$ and $[pz]$ meets $F$ transversally;
\item the function $d_M(p,\cdot)$ is differentiable at $z$;
\item the function $d_M(\cdot,z)$ is differentiable at $p$.
\end{enumerate}
\end{lemma}

\begin{proof}
The proof is similar to that of Lemma 2.14 in \cite{KKL} and its essence
is explained above. Here are the formal details.

In order to use standard properties of Riemannian cut loci, we extend
$M$ beyond the boundary to obtain
a complete boundaryless Riemannian $n$-manifold $\widehat M$
such that $F$ is a smooth hypersurface in $\widehat M$
separating $M$ from its complement.
Since $q$ is a nearest to $p$ point of $F$,
$[pq]$ is a unique
shortest path between $p$ and $q$ in both $M$ and $\widehat M$.
By \cite[Lemma 2.13]{KKL}, $p$ and $q$ are not conjugate along $[pq]$.

Therefore $q$ is not a cut point of $p$ in $\widehat M$.
Hence there is a neighborhood $W$ of $q$ in $\widehat M$
such $d_{\widehat M}(p,\cdot)$ is smooth on $W$ and
every point $z\in W$ is connected to $p$
by a unique $\widehat M$-minimizing geodesic whose
direction at $z$ depends smoothly on~$z$.
Since $[pq]$ meets $F$ orthogonally at $q$
and has no other points on $F$,
one can choose a neighborhood
$W_0\subset W$ of $q$ such that, for every $z\in W_0$ the $\widehat M$-minimizing
geodesic $[pz]$ intersects $F$ at most once and transversally.
For $z$ from the ``half-neighborhood'' 
$W_0\cap M$, this property implies that $[pz]\subset M$
and therefore $[pz]$ is a shortest path in both $\widehat M$ and $M$.
Hence $d_M(p,z)=d_{\widehat M}(p,z)$ for all $z\in W_0\cap M$.

Thus for all $z\in V:=W_0\cap F$ the requirements (1)--(3) of the lemma are satisfied.
To prove (4), recall that the cut-point relation is
symmetric.
Hence, for every $z\in V$, $p$ is not a cut point of $z$
and a similar argument shows that 
$d_M(\cdot,z)=d_{\widehat M}(\cdot,z)$ in a neighborhood of $p$.
These properties imply that $d_M(\cdot,z)$ is differentiable at~$p$.
\end{proof}

\begin{remark}\label{r:twice good}
By Lemma \ref{l:phi-nearest}, a nearest boundary point 
$y$ to $x\in U$ is also a nearest boundary
point to $\phi(x)$ in $M'$.
Applying Lemma \ref{l:no-cut} to both manifolds
and taking the intersection of the respective neighborhoods,
we obtain a neighborhood $V\subset F$ satisfying
the requirements (1)--(4) of the lemma
for both $x\in M$ and $\phi(x)\in M'$. 
\end{remark}

\section{Shortest paths to boundary points}
\label{sec:shortest}

The main result of this section is Lemma \ref{l:phi-segment}
about $\phi$-images of certain geodesics
(compare with \cite[Lemma 2.9]{LS} and \cite[Lemma 6.3]{I20}).
The following preparation lemma 
characterizes these geodesics in terms of distance difference
functions.

\begin{lemma}\label{l:segment}
Let $p\in M\setminus F$
and let $z\in F$ be a point satisfying
conditions (1)--(4) from Lemma \ref{l:no-cut}.
Then for every $x\in M$ the following holds:
$x\in[pz]$ if and only if the function
$\Phi\co F\to\R$ given by
\be\label{e:Dpx}
 \Phi(y) = D_p(y,z)-D_x(y,z), \qquad y\in F,
\ee
where $D_p=\D(p)$ and $D_x=\D(x)$
(see \eqref{e:ddef} and \eqref{e:ddef2})
attains its maximum at~$z$.
\end{lemma}

\begin{proof}
Substituting \eqref{e:ddef} into \eqref{e:Dpx} yields that
$\Phi(y)=\Psi(y)+C$ where
$$
 \Psi(y) = d_M(p,y) - d_M(x,y)
$$
and $C=d_M(x,z)-d_M(p,z)$ does not depend on~$y$.
We prove the statement of lemma for $\Psi$ instead of $\Phi$.
The statements for $\Phi$ и $\Psi$ are equivalent since
the two functions have the same points of maxima.

To prove the ``only if'' part, consider $x\in[pz]$.
For all $y\in F$ we have
$$
 \Psi(y) = d_M(p,y) - d_M(x,y) \le d_M(p,x)
$$
by the triangle inequality.
Since $x\in[pz]$, this inequality turns into equality for $y=z$.
Hence 
$$
\Psi(z) = d_M(p,x) = \max_{y\in F} \Psi(y) .
$$
Thus $z$ is a point of maximum of $\Psi$.

To prove the ``if'' part, consider $x\in M$ and assume that
$\Psi$ attains its maximum at~$z$.
First we show that $\ovr{zx}=\ovr{zp}$.
Suppose the contrary.
Let $v\in T_zF$ be the orthogonal projection of 
the vector $\ovr{zx}-\ovr{zp}$ to $T_zF$.
Since $\ovr{zx}$ and $\ovr{zp}$ belong to the hemisphere
of $S_zM$ bounded by the hyperplane $T_zF\subset T_zM$,
these two vectors have different projections to $T_zF$.
Hence $v\ne 0$ and moreover
\be\label{e:segment2}
 \< v, \ovr{zx}-\ovr{zp} \> > 0 .
\ee
Let $\ga\co[0,\ep)\to F$ be a smooth curve with $\ga(0)=z$
and $\dot\ga(0)=v$.
By our assumptions the function $d_M(p,\cdot)$ is
differentiable at $z$, hence by \eqref{e:grad},
\be\label{e:segment1}
 \frac d{dt} d_M(p,\ga(t))\big|_{t=0} = - \< v, \ovr{zp} \> .
\ee
Construct a smooth variation of curves $\{\sigma_t\}$, $t\in[0,\ep)$,
where $\sigma_0=[xz]$ and $\sigma_t$ connects $x$ to $\ga(t)$ for every~$t$.
By the first variation formula,
$$
 \frac d{dt}\operatorname{length}(\sigma_t) = - \< v, \ovr{zx} \> .
$$
Hence
$$
 d_M(x,\ga(t)) \le \operatorname{length}(\sigma_t)
 \le d_M(x,z) -t \< v, \ovr{zx} \> + o(t), \qquad t\to 0.
$$
This and \eqref{e:segment1} imply that
$$
 \Psi(\ga(t)) = d_M(p,\ga(t)) - d_M(x,\ga(t))
 \ge \Psi(z) + t \< v, \ovr{zx}-\ovr{zp} \> + o(t), \qquad t\to 0.
$$
By \eqref{e:segment2}, this implies that
$
 \Psi(\ga(t)) >  \Psi(z)
$
for a sufficiently small $t>0$.
Hence $\Psi(z)$ is not a maximum of $\Psi$, a contradiction.

This contradiction shows that $\ovr{zx}=\ovr{zp}$,
hence either $x\in [zp]$ or $p\subset[zx]$.
It remains to rule out the latter case.
Suppose that $p\subset[zx]$.
Then we can repeat the argument of the ``only if'' part
with $p$ and $z$ swapped. Namely, for all $y\in F$,
$$
 \Psi(y) = d_M(p,y) - d_M(x,y) \ge -d_M(p,x)
$$
by the triangle inequality. This inequality turns into equality
only for $y=z$ since $[pz]$ intersects $F$ only at $z$
and this intersection is transversal.
Thus $\Psi(z)$ is a strict minimum of $\Psi$ rather than
the maximum, a contradiction.
This finishes the proof of the ``if'' part and of the lemma.
\end{proof}

Now we are in a position to prove the main result of this section.

\begin{lemma}\label{l:phi-segment}
Let $p\in U\setminus F$, $p'=\phi(p)$ and let $V\subset F$ be a neighborhood
constructed in Remark \ref{r:twice good}. Then for every $z\in V$,
$$
 \phi([pz]\cap U) = [p'z]\cap U'
$$
where the shortest paths in the left- and right-hand side
are in $M$ and $M'$, resp.
\end{lemma}

\begin{proof}
Since $\phi$ is a bijection between $U$ and $U'$,
we can reformulate the lemma as follows:
a point $x\in U$ belongs to $[pz]$ if and only if $\phi(x)$
belongs to $[p'z]$.
By Lemma \ref{l:segment},
$x\in[pz]$ if and only if $z$ is a point of maximum of the function
\be\label{e:phiseg1}
  \Phi(y) = D_p(y,z)-D_x(y,z), \qquad y\in F.
\ee
on $F$.
By the same lemma applied to $M'$, $\phi(x)\in[p'z]$ if and only if
$z$ is a point of maximum of the function
\be\label{e:phiseg2}
 \Phi'(y) =  D'_{\phi(p)}(y,z)-D'_{\phi(x)}(y,z), \qquad y\in F.
\ee
By \eqref{e:phi0} we have $\Phi=\Phi'$, hence the two
maximality conditions are equivalent.
\end{proof}

\section{Derivative of $\phi$ and proof of the theorems}
\label{sec:final}

Recall that $\phi\co U\to U'$ is a locally bi-Lipschitz homeomorphism.
By Rademacher's theorem, every locally Lipschitz is differentiable
almost everywhere.
Applying this to $\phi$ and $\phi^{-1}$ yields that
$\phi$ is differentiable a.e.\ and 
its differential $d_x\phi$ at any differentiability point $x\in U$
is a non-degenerate linear map from $T_xM$ to $T_{\phi(x)}M'$.
In the next key lemma we show that this differential is an isometry.

\begin{lemma}\label{l:dphi}
Let $p\in U$ be a point where $\phi$ is differentiable and $p'=\phi(p)$.
Then the differential $d_p\co T_pM\to T_{p'}M'$ is a linear isometry
with respect to $g$ and $g'$.
\end{lemma}

\begin{proof}
Let $V\subset F$ be a neighborhood constructed in Remark \ref{r:twice good}.
Then every point $z\in V$ satisfies conditions (1)--(4)
of Lemma \ref{l:no-cut} for both $p$ in $M$ and $p'$ in~$M'$.

These conditions imply that
for every $z\in V$, there is a unique shortest path $[pz]$ and
it initial direction $\ovr{pz}$ depends continuously on~$z$.
Hence the map $z \mapsto \ovr{pz}$ is a homeomorphism
from $V$ onto an open subset $\Sigma$ of the sphere $S_pM$.

Pick two different vectors $v_1,v_2\in\Sigma$
and let $z_1,z_2\in V$ be such that $v_i=\ovr{pz_i}$, $i=1,2$.
Let $\ga_i$, $i=1,2$, denote the unit-speed parametrization of $[pz_i]$
with $\ga_i(0)=p$.
By the choice of $V$ (see Lemma \ref{l:no-cut}(4)) the function
$$
 t\mapsto D_{\ga_1(t)}(z_1,z_2) = d_M(\ga_1(t),z_1)-d_M(\ga_1(t),z_2)
$$
is differentiable at $t=0$, and by \eqref{e:grad} its derivative is given by
\be\label{e:dphi1}
 \frac d{dt} D_{\ga_1(t)}(z_1,z_2)\big|_{t=0} = -1 + \<v_1,v_2\> .
\ee
Similarly (swapping $v_1$ and $v_2$),
\be\label{e:dphi2}
 \frac d{dt} D_{\ga_2(t)}(z_2,z_1)\big|_{t=0} = -1 + \<v_1,v_2\> .
\ee
The scalar products above are $g$-products in $T_pM$.

Now consider the images of these curves and vectors under $\phi$ and $d_p\phi$.
For $i=1,2$, define 
$$
\la_i=\|d_p\phi(v_i)\|_{g'}
$$ 
and 
$$
w_i=\frac{d_p\phi(v_i)}{\la_i} .
$$
Let $\ep>0$ be such that $\ga_i([0,\ep))\subset U$ for $i=1,2$.
Then by Lemma \ref{l:phi-segment}, $\phi\circ\ga_i|_{[0,\ep)}$ 
parametrizes an initial interval of $[p'z_i]$.
The velocity of $\phi\circ\ga_i|_{[0,\ep)}$ at 0 equals
$d_p\phi(v_i)=\la_i w_i$, hence $w_i = \ovr{p'z_i}$, $i=1,2$.

Now similarly to \eqref{e:dphi1} we calculate the derivative
\be\label{e:dphi3}
 \frac d{dt} D'_{\phi(\ga_1(t))}(z_1,z_2)\big|_{t=0} = \la_1(-1 + \<w_1,w_2\>) .
\ee
By \eqref{e:phi0}, the functions differentiated in \eqref{e:dphi1} and \eqref{e:dphi3} are the same, hence
\be\label{e:dphi4}
-1 + \<v_1,v_2\> = \la_1(-1 + \<w_1,w_2\>)
\ee
where $\<w_1,w_2\>$ is the scalar product with respect to $g'$.
Similarly from \eqref{e:dphi2} we obtain that
\be\label{e:dphi5}
-1 + \<v_1,v_2\> = \la_2(-1 + \<w_1,w_2\>) .
\ee
By \eqref{e:dphi4} and \eqref{e:dphi5},
$$
\la_1(-1 + \<w_1,w_2\>) = \la_2(-1 + \<w_1,w_2\>) ,
$$
therefore $\la_1=\la_2$
(note that $\<w_1,w_2\> \ne 1$ since $w_1$ and $w_2$ are different unit vectors).
Substituting the definitions of $\la_1$ and $\la_2$ we obtain that
$$
\|d_p\phi(v_1)\|_{g'}=\|d_p\phi(v_2)\|_{g'} .
$$
Since $v_1$ and $v_2$ are arbitrary vectors from $\Sigma$,
this identity implies that 
the function $v\mapsto \|d_p\phi(v)\|_{g'}$ is constant on $\Sigma$.
We denote this constant by~$\la$.
Since $\Sigma$ is an open subset of the sphere $S_pM$,
it follows that $d_p\phi$ is a $\la$-homothetic linear map:
$$
  \|d_p\phi(v)\|_{g'} = \la \|v\|_{g}
$$
for all $v\in T_pM$. 
Hence $d_p\phi$ preserves the angles, in particular $\<v_1,v_2\>=\<w_1,w_2\>$.
Now \eqref{e:dphi4} implies that $\la=\la_1=1$.
Thus $d_p\phi$ is an isometry.
\end{proof}

\subsection*{Proof of the theorems \ref{t:whole} and \ref{t:part}}

As shown in Lemma \ref{l:dphi}, the derivative of our bi-Lipschitz
homeomorphism $\phi$ is an isometry almost everywhere.
Hence $\phi$ is a 1-Lipschitz map, i.e.\
it does not increase arclength distances.
The same holds for $\phi^{-1}$, therefore $\phi$ is a distance isometry.
By the Myers-Steenrod theorem (\cite{MS}, see also \cite[Ch.~5, Theorem 18]{Pe})
every distance isometry between Riemannian manifolds
is a smooth Riemannian isometry.
Thus $\phi$ is a Riemannian isometry.
Lemma \ref{l:id-on-boundary} implies the last claim of Theorem \ref{t:part}
and this finishes the proof of Theorem~\ref{t:part}.
As explained in the introduction, Theorem \ref{t:part} implies
Theorem \ref{t:whole}.

\end{document}